\documentclass[12pt]{amsart}

\usepackage{amsmath}
\usepackage{amscd}
\usepackage{amssymb}
\usepackage{enumerate}
\usepackage{amsfonts}
\usepackage{graphicx}
\usepackage[all]{xy}
\usepackage{mathrsfs}
\usepackage[hypertex]{hyperref}

\usepackage{bbm}

\newcommand{\limto}{{\displaystyle\lim_{\longrightarrow}}}
\newcommand{\rightlim}{\mathop{\limto}}

\newcommand{\leftlim}{\mathop{\displaystyle\lim_{\longleftarrow}}}
\newcommand{\limfromn}{\leftlim\limits_{\raise3pt\hbox{$n$}}}
\newcommand{\limton}{\rightlim\limits_{\raise3pt\hbox{$n$}}}

\newcommand{\rightlimit}[1]{\mathop{\lim\limits_{\longrightarrow}}\limits%
                   _{\raise3pt\hbox{$\scriptstyle #1$}}}
\newcommand{\leftlimit}[1]{\mathop{\lim\limits_{\longleftarrow}}\limits%
                   _{\raise3pt\hbox{$\scriptstyle #1$}}}

\numberwithin{equation}{section}

\newcommand{\rar}[1]{\stackrel{#1}{\longrightarrow}}
\newcommand{\xrar}[1]{\xrightarrow{#1}}

\newcommand{\into}{\hookrightarrow}

\newcommand{\al}{\alpha}
\newcommand{\be}{\beta}

\newcommand{\Ga}{\Gamma}

\newcommand{\la}{\lambda}

\newcommand{\ze}{\zeta}

\newcommand{\vp}{\varphi}

\newcommand{\bC}{{\mathbb C}}

\newcommand{\bF}{{\mathbb F}}
\newcommand{\bG}{{\mathbb G}}
\newcommand{\bH}{{\mathbb H}}

\newcommand{\bN}{{\mathbb N}}

\newcommand{\bQ}{{\mathbb Q}}
\newcommand{\bR}{{\mathbb R}}

\newcommand{\bZ}{{\mathbb Z}}

\newcommand{\cH}{{\mathcal H}}

\newcommand{\cR}{{\mathcal R}}

\newcommand{\fg}{{\mathfrak g}}
\newcommand{\fh}{{\mathfrak h}}

\newcommand{\fk}{{\mathfrak k}}

\newcommand{\Hom}{\operatorname{Hom}}

\newcommand{\id}{\operatorname{id}}

\newcommand{\Span}{\operatorname{span}}

\newcommand{\Ind}{\operatorname{Ind}}

\newcommand{\Rep}{\operatorname{Rep}}

\newcommand{\tens}{\otimes}

\newcommand{\st}{\,\big\vert\,}

\newcommand{\sbr}{\smallbreak}
\newcommand{\mbr}{\medbreak}

\newtheorem{thm}{Theorem}[section]
\newtheorem{cor}[thm]{Corollary}
\newtheorem{lem}[thm]{Lemma}

\newtheorem{prop}[thm]{Proposition}

\theoremstyle{remark}
\newtheorem{rem}[thm]{Remark}

\newtheorem{defin}[thm]{Definition}

\newcommand{\cst}{\bC^\times}

\newcommand{\cInd}{c\!-\!\operatorname{Ind}}
\newcommand{\uInd}{u\!-\!\operatorname{Ind}}

\begin{document}

\title[Unipotent groups over local fields]{Representations of unipotent groups over local fields and Gutkin's conjecture}

\author{Mitya Boyarchenko}
\email{mityab@umich.edu}
\address{Mathematics Department \\
University of Michigan \\
530 Church Street \\
Ann Arbor, MI 48109--1043, USA}


\begin{abstract}
Let $F$ be a finite field or a local field of \emph{any} characteristic. If $A$ is a finite dimensional associative nilpotent algebra over $F$, the set $1+A$ of all formal expressions of the form $1+x$, where $x\in A$,
is a locally compact group with the topology induced by the standard one on $F$ and the multiplication $(1+x)\cdot(1+y)=1+(x+y+xy)$. We prove a result conjectured by E.~Gutkin in 1973: every
unitary irreducible representation of $1+A$ can be obtained by
unitary induction from a $1$-dimensional unitary character of a subgroup of the form $1+B$, where $B\subset A$ is an $F$-subalgebra.
In the case where $F$ is local and nonarchimedean we also establish an analogous result for \emph{smooth} irreducible
representations of $1+A$ over $\bC$ and
show that every such representation is admissible and carries an
invariant Hermitian inner product.
\end{abstract}

\maketitle

\section{Introduction}\label{s:introduction}

Let $F$ be a self-dual field, that is, a finite field, or $\bR$, or
$\bC$, or a finite extension of $\bQ_p$, or a field $\bF_q((t))$ of
formal Laurent series in one variable over a finite field. We equip
$F$ with the natural topology, making it a locally compact
topological field (for a finite field, the topology is discrete).
The term ``self-dual'' is explained by the following observation.
Fix a nontrivial unitary character $\psi:(F,+)\rar{}\cst$ of the
additive group of $F$. For each $a\in F$, define
$\psi_a(x)=\psi(ax)$ for all $x\in F$. Then the map $a\mapsto\psi_a$
is a topological isomorphism between $(F,+)$ and its Pontryagin
dual.

\mbr

Let $A$ be an associative $F$-algebra, which is not assumed to be
unital. For each $m\geq 1$, write $A^m$ for the subspace of $A$
spanned by all elements of the form $a_1 a_2\dotsm a_m$, where
$a_j\in A$. We say that $A$ is \emph{nilpotent} if $A^m=0$ for some
$m\geq 1$. In this case the set $1+A$ of all formal expressions of
the form $1+x$, where $x\in A$, is a group under the operation
$(1+x)(1+y)=1+(x+y+xy)$. (For instance, if $A^2=0$, then $1+A$ is identified with the additive group of $A$.)

\mbr

Now suppose that $A$ is a finite dimensional associative nilpotent
algebra over $F$. Then $A$ inherits a natural topology from $F$, and
$1+A$ becomes a locally compact (Hausdorff) and second countable topological group.
Moreover, it is unimodular. If $B$ is an $F$-subalgebra of $A$, then
$1+B$ can be viewed as a closed subgroup of $1+A$.

\mbr

In this situation, E.~Gutkin formulated in \cite{gutkin} the
following statement:

\begin{thm}\label{t:gutkin-unitary}
Let $\pi:1+A\rar{}U(\cH)$ be a unitary irreducible representation.
Then there exist an $F$-subalgebra $B\subset A$ and a unitary
character $\al:1+B\rar{}\cst$ such that
$\pi\cong\uInd_{1+B}^{1+A}\al$.
\end{thm}

In this statement, $U(\cH)$ denotes the group of unitary operators
on a complex Hilbert space $\cH$, and $\uInd$ denotes the operation
of unitary induction (see \S\ref{sss:unitary-induction}).

\mbr

The original motivation behind Theorem \ref{t:gutkin-unitary} was
the following application:

\begin{cor}\label{c:dimensions}
If $F=\bF_q$ is a finite field with $q$ elements and $A$ is a finite
dimensional associative nilpotent algebra over $\bF_q$, the
dimension of every complex irreducible representation of the finite
group $1+A$ is a power of $q$.
\end{cor}

This result follows from Theorem \ref{t:gutkin-unitary}, because, in
the situation of Corollary \ref{c:dimensions}, if $B\subset A$ is a
subalgebra, then the index of $1+B$ in $1+A$ equals
$q^{\dim(A)-\dim(B)}$, which is a power of $q$.

\mbr

As an example, fix an integer $n\geq 2$ and let $A$ be the algebra
of strictly upper-triangular matrices of size $n$ over $\bF_q$
(under the usual operations of matrix addition and matrix
multiplication). Then $1+A$ can be identified with the group
$UL_n(\bF_q)$ of unipotent upper-triangular matrices of size $n$
over $\bF_q$. In 1960 G.~Higman asked whether the dimension of every (complex) irreducible representation of $UL_n(\bF_q)$ is a power of $q$ (thus Corollary \ref{c:dimensions} yields an affirmative answer). Later this question was advertised and popularized by J.~Thompson and A.A.~Kirillov, among others, and eventually it led Gutkin to introduce the more general groups of the form $1+A$ and formulate the more precise Theorem \ref{t:gutkin-unitary} in \cite{gutkin}.

\mbr

To the best of our knowledge, the first complete proof of Corollary
\ref{c:dimensions} was given by I.M.~Isaacs \cite{isaacs} in 1995, whereas the first complete proof of Theorem \ref{t:gutkin-unitary} when $F$
is a \emph{finite} field was given by Z.~Halasi \cite{halasi} in
2004. Unfortunately, Halasi's proof relies on the result of
\cite{isaacs}; in particular, it uses a counting argument, and it is
not clear how to adapt this argument to the case where $F$ is a
local field.

\mbr

When $F$ is a local field of characteristic $0$ (that is, $\bR$ or
$\bC$ or a finite extension of $\bQ_p$), one can prove Theorem
\ref{t:gutkin-unitary} using the orbit method for unipotent groups
over $\bR$ \cite{kirillov} or over $\bQ_p$ \cite{moore}. We briefly
explain the idea of this approach in \S\ref{ss:local-char-zero}.

\mbr

The first goal of this article is to give a proof of Theorem
\ref{t:gutkin-unitary} for an arbitrary self-dual field $F$, in a
way that is independent of the characteristic of $F$ (see \S\ref{ss:proof-unitary} and Remark \ref{r:generality}). An important
ingredient in the proof is a certain result on commutators in groups
of the form $1+A$ that was first proved in \cite{halasi}. We review
it in \S\ref{ss:free-nilpotent}.

\mbr

One can also formulate Gutkin's conjecture in a different setting.
Namely, suppose that $F$ is a nonarchimedean local field. Then the
topological group $1+A$ is totally disconnected: there exists a
basis of neighborhoods of the identity element of $1+A$ consisting
of compact open subgroups. Thus it is natural to study \emph{smooth}
complex representations of $1+A$. Our second goal is to prove the
following result:

\begin{thm}\label{t:gutkin-smooth}
Let $F$ be a nonarchimedean local field, and let $\pi:1+A\to GL(V)$ be a
smooth irreducible representation. Then $\pi$ is admissible and
unitarizable. Moreover, there exist an $F$-subalgebra $B\subset A$
and a smooth homomorphism $\al:1+B\rar{}\cst$ such that
$\pi\cong\Ind_{1+B}^{1+A}\al$ and the natural map
$\cInd_{1+B}^{1+A}\al\rar{}\Ind_{1+B}^{1+A}\al$ is an isomorphism.
\end{thm}

In this statement, $GL(V)$ denotes the group of all linear
automorphisms of a vector space $V$ over $\bC$, and $\Ind$ and
$\cInd$ denote the operations of smooth induction and smooth induction with compact supports, respectively. All the terminology and notation used in Theorem \ref{t:gutkin-smooth} is reviewed in Section \ref{s:smooth-induction}. The theorem itself is proved in \S\ref{ss:proof-smooth}.

\mbr

It is natural to ask whether, if $\bG$ is any unipotent algebraic
group over a nonarchimedean local field $F$, then every smooth
complex irreducible representation of $\bG(F)$ is admissible and
unitarizable. At present an affirmative answer is known when
$\operatorname{char}F=0$ thanks to \cite{rodier,van-dijk}. The case
$\operatorname{char}F>0$ is open.

\mbr

All the main ideas of our approach to Theorems
\ref{t:gutkin-unitary} and \ref{t:gutkin-smooth} can already be
explained in the case where $F$ is a finite field; at the same time,
in this setting one avoids various technical complications that
arise in the study of infinite dimensional (smooth or unitary)
representations. Thus, for clarity, we treat this case first in
\S\ref{ss:finite-fields-general}.

\mbr

Section \ref{s:auxiliary} is essentially independent of, and different in spirit from, the rest of the article. In it we apply the techniques developed by Halasi in \cite{halasi} to prove Corollary \ref{c:commutator-pairing}, which is the key algebraic input in our proofs of Theorems \ref{t:gutkin-unitary} and \ref{t:gutkin-smooth}.

\mbr

In Section \ref{s:smooth-induction} we review the basics of $p$-adic representation theory and conclude by stating two theorems due to F.~Rodier \cite{rodier} that are crucial for our arguments. Finally, the complete proofs of Theorems \ref{t:gutkin-unitary} and \ref{t:gutkin-smooth} are presented in Section \ref{s:proof}.

\mbr

\noindent \textit{Acknowledgements.} My research was supported by the
NSF grant DMS-0701106 and the NSF Postdoctoral Research Fellowship
DMS-0703679. I thank Eugene Gutkin for inciting my interest in his
conjecture, and Akaki Tikaradze for many stimulating discussions. I
am grateful to Vladimir Drinfeld, from whom I learned the general
strategies for studying representations of nilpotent groups that are
employed in the article. I was also inspired by recent work of Jeff
Adler and Alan Roche \cite{adler-roche} on discrete series representations of unipotent groups over local fields.

\section{Some special cases of Gutkin's conjecture}\label{s:special-cases}

We begin this section by introducing some general notation. Then we proceed to explain how Theorem \ref{t:gutkin-unitary} can be proved, in the case where $F$ is a local field of characteristic $0$, using the classical orbit method due to Kirillov in the case where $F$ is archimedean and to C.C.~Moore in the case where $F$ is nonarchimedean. Finally we sketch a proof of Theorem \ref{t:gutkin-unitary} in the case where $F$ is finite. The proof is different from the ones that can be found in the literature \cite{halasi}, and contains all of the new ideas that allow us to prove Theorems \ref{t:gutkin-unitary} and \ref{t:gutkin-smooth} in full generality.

\subsection{Notation} Throughout this paper, $(\cdot,\cdot)$
denotes the group commutator and $[\cdot,\cdot]$ denotes the Lie
algebra commutator. We enunciate this in
\S\S\ref{sss:group-commutators}--\ref{sss:lie-brackets}.

\subsubsection{}\label{sss:group-commutators} Let $\Ga$ be a group.  If $S_1,S_2\subset\Ga$
are subsets, then $(S_1,S_2)$ denotes the subgroup of $\Ga$
generated by all elements of the form $ghg^{-1}h^{-1}$ with $g\in S_1$ and
$h\in S_2$.

\subsubsection{}\label{sss:lie-brackets} Let $A$ be an associative (not necessarily unital) ring. If
$x,y\in A$, then $[x,y]\overset{\text{def}}{=}xy-yx$. If
$S_1,S_2\subset A$ are subsets, then $[S_1,S_2]$ denotes the
additive subgroup of $A$ generated by all elements of the form
$[x,y]$ with $x\in S_1$ and $y\in S_2$.

\subsubsection{} If $A$ is an associative ring, the two-sided ideals
$A^m\subset A$ for all $m\geq 1$ and the notion of what it means for
$A$ to be nilpotent are defined as in the Introduction.

\subsubsection{} Let $A$ be a nilpotent associative ring. The
corresponding group $1+A$ is defined as in the Introduction. If
$I\subset A$ is a subring, then $1+I$ is naturally identified with a
subgroup of $1+A$. If $I$ is a two-sided ideal of $A$, then $1+I$ is
a normal subgroup of $1+A$ (the converse is usually false), and the
quotient group $(1+A)/(1+I)$ can be naturally identified with
$1+(A/I)$, where $A/I$ is the corresponding quotient ring. All these
remarks are used implicitly in what follows.

\subsection{Local fields of characteristic
$0$}\label{ss:local-char-zero}
Let us first settle Gutkin's original conjecture (Theorem \ref{t:gutkin-unitary}) in the case where the base field is $\bR$ or $\bC$.
We will use the classical orbit method due to Kirillov \cite{kirillov} together with an observation that we originally learned from a work of C.A.M.~Andr\'e \cite{andre}.

\mbr

Let $F$ be a local archimedean field, and let $A$ be a finite dimensional associative nilpotent algebra over $F$. There is no loss of generality in assuming that $F=\bR$ (if $F=\bC$ we can view $A$ as an algebra of dimension $2\cdot\dim_{\bC}A$ over $\bR$). We can view $1+A$ as a nilpotent Lie group, and its Lie algebra can be naturally identified with $A$ (where the latter is equipped with the commutator bracket $[a,b]=ab-ba$) by means of the exponential map:
\[
\exp : A \rar{\sim} 1+A, \qquad x\longmapsto 1+x+\frac{x^2}{2!}+\frac{x^3}{3!}+\dotsb
\]
(here we use both the fact that $\bR$ has characteristic $0$ and the assumption that $A$ is nilpotent).

\mbr

In general, if $G$ is a nilpotent Lie group, $\fg$ is its Lie algebra, and $f:\fg\rar{}\bR$ is a linear map, a \emph{polarization} of $f$ is a Lie subalgebra $\fh\subset\fg$ that satisfies $f([\fh,\fh])=0$ and has the maximal possible dimension subject to this property. Kirillov proved in \cite{kirillov} that every linear $f:\fg\rar{}\bR$ has a polarization $\fh$, and the corresponding induced representation $\uInd_H^G\chi_f$ of $G$ is irreducible. Here $H\subset G$ is the subgroup corresponding to $\fh$ and $\chi_f=\exp(if)$ is the unitary character of $H$ defined by $f$. Moreover, Kirillov showed that every unitary irreducible representation of $G$ is unitarily equivalent to $\uInd_H^G\chi_f$ for suitable $f$ and $\fh$.

\mbr

In view of these results, we see that Theorem \ref{t:gutkin-unitary} for $F=\bR$ follows from

\begin{lem}\label{l:andre}
Let $A$ be a finite dimensional associative nilpotent algebra over a field $\fk$, and let $f:A\rar{}\fk$ be a linear functional. Then there exists an $($associative$)$ $\fk$-subalgebra $B\subset A$ such that, as a Lie subalgebra of $A$ with respect to the commutator bracket, $B$ is a polarization of $f$.
\end{lem}
For the proof, see \cite{andre} or \cite[Theorem D.13]{intro}.

\mbr

Next let us consider Theorem \ref{t:gutkin-unitary} in the case where $F$ is a finite extension of $\bQ_p$. As before, we may assume that $F=\bQ_p$, and then we only need to repeat the proof of the case $F=\bR$ almost verbatim, replacing the reference to \cite{kirillov} with a reference to \cite{moore}, where the orbit method for nilpotent Lie groups over $\bQ_p$ was developed.

\begin{rem}
We were somewhat sketchy in this treatment of Theorem \ref{t:gutkin-unitary} for local fields of characteristic $0$ because the argument we give in \S\ref{ss:proof-unitary} below is also valid in this setting, and provides a different approach that does not directly rely on the results of \cite{kirillov,moore}. Of course, some of the ideas we use are still directly inspired by the works of G.~Mackey, as well as Kirillov, Moore and the other pioneers of the orbit method for nilpotent groups.
\end{rem}

\subsection{Finite fields}\label{ss:finite-fields-general}
Let us conclude the section by sketching a proof of Theorem \ref{t:gutkin-unitary} in the case where $F$ is a finite field. We will be brief, because Section \ref{s:proof} below presents a self-contained proof of Theorems \ref{t:gutkin-unitary} and \ref{t:gutkin-smooth} that is valid both for finite fields and for local nonarchimedean fields. On the other hand, the argument we sketch here already contains all of the important ideas behind the proofs of Theorems \ref{t:gutkin-unitary} and \ref{t:gutkin-smooth} in full generality. Therefore we would like to present these ideas in an elementary context before reviewing the technical machinery of Section \ref{s:smooth-induction}.

\mbr

The proof is based on the following crucial result.

\begin{cor}\label{c:commutator-pairing}
Let $\fk$ be an arbitrary field, and let $A$ be an associative
nilpotent algebra over $\fk$. Fix an integer $m\geq 2$, let
$\ze:1+A^m\rar{}\cst$ be a homomorphism that is invariant under the
conjugation action of $1+A$ on $1+A^m$, and define
\[
C_\ze : (1+A)\times (1+A^{m-1}) \rar{} \cst
\]
by $C_\ze(g,h)=\ze(ghg^{-1}h^{-1})$. Then
 \sbr
\begin{enumerate}[$($a$)$]
\item $C_\ze$ factors through a map
\[
\overline{C}_\ze: \bigl( 1+(A/A^2) \bigr) \times \bigl(
1+(A^{m-1}/A^m)\bigr) \rar{} \cst;
\]
and
 \sbr
\item the map $\overline{C}_\ze$ is $\fk$-bilinear in the following sense:
 \sbr
\begin{itemize}
\item
$\overline{C}_\ze(1+x_1+x_2,1+y)=\overline{C}_\ze(1+x_1,1+y)\cdot\overline{C}_\ze(1+x_2,1+y)$
for all $x_1,x_2\in A/A^2$ and all $y\in A^{m-1}/A^m$;
 \sbr
\item
$\overline{C}_\ze(1+x,1+y_1+y_2)=\overline{C}_\ze(1+x,1+y_1)\cdot\overline{C}_\ze(1+x,1+y_2)$
for all $x\in A/A^2$ and all $y_1,y_2\in A^{m-1}/A^m$;
 \sbr
\item $\overline{C}_\ze(1+\la x,1+y)=\overline{C}_\ze(1+x,1+\la y)$ for all
$\la\in\fk$, $x\in A/A^2$, $y\in A^{m-1}/A^m$.
\end{itemize}
\end{enumerate}
\end{cor}

This corollary follows immediately from Proposition \ref{p:commutator-pairing} below\footnote{Observe that $\ze$ is invariant under $(1+A)$-conjugation if and only if $\ze$ is trivial on the commutator $(1+A,1+A^m)$.}, which is proved in Section \ref{s:auxiliary} (the proof we give there is independent of any of the other results of this article).

\mbr

Now let $F$ be a finite field, let $A$ be a finite dimensional associative nilpotent $F$-algebra, and let $\pi:1+A\rar{}GL(V)$ be an irreducible representation of the finite group $1+A$, where $V$ is a finite dimensional vector space over $\bC$. Since $A$ is nilpotent, there exists a smallest integer $m\geq 1$ with the property that $\pi(1+A^m)$ consists of scalar operators. We may assume that $m\geq 2$, since otherwise $\pi$ is already $1$-dimensional and there is nothing to prove.

\mbr

We obtain a homomorphism $\ze:1+A^m\rar{}\cst$ such that $\pi(g)=\ze(g)\cdot\id_V$ for all $g\in 1+A^m$. The construction of $\ze$ implies that it is invariant under the conjugation action of $1+A$. Hence Corollary \ref{c:commutator-pairing} can be applied. Let $\overline{C}_\ze$ be defined as in that corollary and note that $\overline{C}_\ze$ is not identically $1$, since otherwise the minimality of $m$ would be contradicted.

\mbr

Since $1+(A/A^2)$ and $1+(A^{m-1}/A^m)$ are naturally identified with the additive groups of $A/A^2$ and $A^{m-1}/A^m$, respectively, we can view $\overline{C}_\ze$ as a homomorphism of abelian groups
\[
A/A^2 \rar{} (A^{m-1}/A^m)^* \overset{\text{def}}{=} \Hom(A^{m-1}/A^m,\cst).
\]
Since $F$ is finite and $\dim_F (A^{m-1}/A^m)<\infty$, we can further identify $(A^{m-1}/A^m)^*$ with the additive group of the $F$-vector space dual of $A^{m-1}/A^m$. With these identifications, the content of Corollary \ref{c:commutator-pairing}(b) is that the homomorphism
\[
\Phi_\ze : A/A^2 \rar{} (A^{m-1}/A^m)^*
\]
is actually a \emph{linear map of $F$-vector spaces}.

\mbr

Since $\Phi_\ze$ is not identically zero by construction, we can find a $1$-dimensional $F$-vector subspace $L\subset A^{m-1}/A^m$ such that the composition
\[
A/A^2 \xrar{\ \ \Phi_\ze\ \ } (A^{m-1}/A^m)^* \rar{} L^*
\]
is surjective, where the second map is given by restriction. Let $A_1\subset A$ denote the preimage of the kernel of the above composition; then $A_1$ is a codimension $1$ two-sided ideal of $A$. Also, let $U\subset A^{m-1}$ denote the preimage of $L$; then $U$ is a two-sided ideal of $A$. It is easy to see that $\ze\bigl\lvert_{(1+U,1+U)}\equiv 1$; or, equivalently, that $U\subset A_1$. (See Lemma \ref{l:annihilate-commutator} for the proof.)

\mbr

Write $G=1+A$, $H=1+A_1$, $N=1+U$. Then $N\subset H\subset G$, and both $N$ and $H$ are normal subgroups of $G$. Let $\rho$ be an irreducible summand of the restriction $\pi\bigl\lvert_H$. If we show that $\Ind_H^G\rho$ is irreducible, Frobenius reciprocity will imply that $\pi\cong\Ind_H^G\rho$ and the proof will be complete\footnote{By induction on $\dim_F A$, we may assume that Gutkin's conjecture holds for $A_1$ in place of $A$.}.

\mbr

To check that $\Ind_H^G\rho$ is irreducible, we apply Mackey's irreducibility criterion. Suppose $g\in G$ is such that $g$ leaves the isomorphism class of $\rho$ invariant. Since $\ze$ annihilates the commutator $(H,N)$ by construction, Schur's lemma implies that $\rho$ acts on $N$ via some scalar $\chi:N\rar{}\cst$. We then necessarily have $\chi\bigl\lvert_{1+A^m}=\ze$, and the definition of $H$ implies that $H$ is precisely the stabilizer of $\chi$ under the conjugation action of $G$ on $N$. Since we assumed that $g$ leaves the isomorphism class of $\rho$ invariant, \emph{a fortiori}, $g$ must leave $\chi$ invariant, whence $g\in H$. This verifies the hypothesis of Mackey's irreducibility criterion (in the special case where one is inducing from a normal subgroup of a finite group), and the proof is complete.

\begin{rem}
The proof sketched above has some ideas in common with the proof of Theorem \ref{t:gutkin-unitary} that was presented in \cite{halasi} in the case where $F$ is a finite field. However, that proof also contains some ingredients that seem to have no analogues in the case where $F=\bF_q((t))$. We were able to circumvent this difficulty by using Corollary \ref{c:commutator-pairing}.
\end{rem}

\section{Commutators in groups of the form $1+A$}\label{s:auxiliary}

This section is essentially independent of the rest of the article, in the sense that only the statement of Proposition \ref{p:commutator-pairing} below is used elsewhere in the text, and the proofs we present here rely exclusively on the results and techniques developed in \cite{halasi}.

\subsection{The commutator pairing} The main goal of this section is
the following

\begin{prop}\label{p:commutator-pairing}
Let $\fk$ be an arbitrary field, and let $A$ be an associative
nilpotent algebra over $\fk$. Fix an integer $m\geq 2$, put
\[
Q=(1+A^m)/(1+A,1+A^m),
\]
and define
\[
C : (1+A)\times (1+A^{m-1}) \rar{} Q
\]
by letting $C(g,h)$ be the image of $ghg^{-1}h^{-1}\in 1+A^m$ in
$Q$. Then
 \sbr
\begin{enumerate}[$($a$)$]
\item $C$ factors through a map
\[
\overline{C}: \bigl( 1+(A/A^2) \bigr) \times \bigl(
1+(A^{m-1}/A^m)\bigr) \rar{} Q;
\]
and
 \sbr
\item the map $\overline{C}$ is $\fk$-bilinear in the following sense:
 \sbr
\begin{itemize}
\item
$\overline{C}(1+x_1+x_2,1+y)=\overline{C}(1+x_1,1+y)\cdot\overline{C}(1+x_2,1+y)$
for all $x_1,x_2\in A/A^2$ and all $y\in A^{m-1}/A^m$;
 \sbr
\item
$\overline{C}(1+x,1+y_1+y_2)=\overline{C}(1+x,1+y_1)\cdot\overline{C}(1+x,1+y_2)$
for all $x\in A/A^2$ and all $y_1,y_2\in A^{m-1}/A^m$;
 \sbr
\item $\overline{C}(1+\la x,1+y)=\overline{C}(1+x,1+\la y)$ for all
$\la\in\fk$, $x\in A/A^2$, $y\in A^{m-1}/A^m$.
\end{itemize}
\end{enumerate}
\end{prop}

This result will be proved in \S\ref{ss:proof-p:commutator-pairing} below after reviewing some preliminaries.

\subsection{Free nilpotent algebras}\label{ss:free-nilpotent}
Let $R$ be an arbitrary
associative, commutative, unital ring, let $X$ be a set, and let
$n\in\bN$. If $A$ is any nilpotent associative algebra over $R$,
we call the {\em nilpotence class} of $A$ the smallest integer $k$
such that $A^k=0$. The {\em free
associative nilpotent algebra of nilpotence class $n$ generated by
the set $X$ over $R$} is denoted by $F_R(n,X)$ and is constructed as follows. First
let $\Span_R(X)$ denote the free $R$-module generated by the set
$X$, and let
\[
T = R\langle X\rangle = \bigoplus_{m\geq 0} \bigl( \Span_R(X)
\bigr)^{\otimes m}
\]
be the free associative unital $R$-algebra on the set $X$. It is
naturally graded, and we will write
\[
T_j = \bigl( \Span_R(X) \bigr)^{\otimes j} \quad \text{and} \quad
T_{\geq m}=\bigoplus_{j\geq m} T_j.
\]
Writing $J=F_R(n,X)$ for simplicity, we define
\[
J = T_{\geq 1} \bigl/ T_{\geq n}.
\]
Thus $J$ inherits a grading from $T$. It is clear that $J$
satisfies a universal property with respect to all set maps
from $X$ to associative nilpotent algebras over $R$ of
nilpotence class $\leq n$.

\begin{prop}\label{p:halasi}
Suppose that $R$ is an integral domain whose field of fractions has characteristic $0$. If $J=F_R(n,X)$ for some set $X$ and some integer $n\geq 1$, then for all
$k\geq 2$, we have
\[
(1+J,1+J)\cap (1+J^k) = (1+J,1+J^{k-1}).
\]
\end{prop}

\begin{proof}
When $R=\bZ$ this is \cite[Lem.~2.5]{halasi}. In general one can check, line-by-line, that all the arguments in \S2 of \emph{op.~cit.} remain valid if one replaces $\bZ$ by an arbitrary integral domain $R$ of characteristic $0$ and $\bQ$ by the field of fractions of $R$.
\end{proof}

\begin{thm}\label{t:commutators}
If $A$ is an arbitrary associative nilpotent ring, then for all
$m,n\in\bN$, we have
\[
(1+A^m,1+A^n) \subseteq (1+A,1+A^{m+n-1}).
\]
\end{thm}

This is precisely Theorem 1.4 in \emph{op.~cit.}

\subsection{Proof of Proposition \ref{p:commutator-pairing}}\label{ss:proof-p:commutator-pairing}

We begin with

\begin{lem}\label{l:auxiliary}
Let $N$ be an associative ring, and let $m\geq 2$ be an integer such
that $N^{m+1}=0$. If $x\in N$ and $y\in N^{m-1}$, then
$(1+x)(1+y)(1+x)^{-1}(1+y)^{-1}=1+[x,y]$ in $1+N$.
\end{lem}
\begin{proof}
We have $xyx=0=yx^2$ in $N$. Using this observation, we calculate
\begin{eqnarray*}
(1+x)(1+y)(1+x)^{-1} &=& 1 + (1+x)y(1+x)^{-1} \\
&=& 1 + (1+x)y(1-x) = 1+y+xy-yx.
\end{eqnarray*}
Since $m\geq 2$, we also have $xy^2=yxy=0$. Thus
\[
(1+y+xy-yx)\cdot(1+y)^{-1}=1+(xy-yx)(1+y)^{-1}=1+xy-yx,
\]
as claimed.
\end{proof}

Now Proposition \ref{p:commutator-pairing}(a) follows from the fact that $(1+A^2,1+A^{m-1})\subset(1+A,1+A^m)$, which is a special case of Theorem \ref{t:commutators}.

\mbr

Let us prove the first assertion of Proposition \ref{p:commutator-pairing}(b).
For the purposes of this argument, let us introduce the notation $\be(g,h)=ghg^{-1}h^{-1}$, where $g,h$ are elements of some group. The first assertion of Proposition \ref{p:commutator-pairing}(b) amounts to the following claim: if $x_1,x_2\in A$ and $y\in A^{m-1}$, then
\begin{equation}\label{e:1}
\be(1+x_1+x_2,1+y)\cdot\be(1+x_1,1+y)^{-1}\cdot\be(1+x_2,1+y)^{-1}\in (1+A,1+A^m).
\end{equation}
To see this, let us write $y=\sum_{j=1}^N a_{1,j} a_{2,j} \dotsm a_{(m-1),j}$
for some elements $a_{i,j}\in A$, where $N$ is a positive integer, $1\leq i\leq m-1$ and $1\leq j\leq N$. (We can find such a presentation by the definition of $A^{m-1}$.) Introduce a set $X$ of formal symbols as follows:
\[
X = \bigl\{ \widetilde{x}_1, \widetilde{x}_2 \bigr\} \cup \bigl\{ \widetilde{a}_{i,j} \bigr\}_{1\leq i\leq m-1,\ 1\leq j\leq N}.
\]
Let $n$ be the nilpotence class of $A$ and consider $J=Fr_{\bZ}(n,X)$, the free associative nilpotent algebra of nilpotence class $n$ generated by the set $X$ over $\bZ$ (see \S\ref{ss:free-nilpotent}).

\mbr

Let $\vp:J\rar{}A$ be the homomorphism of associative rings that takes $\widetilde{x}_i\mapsto x_i$ ($i=1,2$) and takes $\widetilde{a}_{i,j}\mapsto a_{i,j}$ for all $i,j$. Then we have $\vp(\widetilde{y})=y$, where
\[
\widetilde{y} = \sum_{j=1}^N \widetilde{a}_{1,j} \widetilde{a}_{2,j} \dotsm \widetilde{a}_{(m-1),j} \in J^{m-1}.
\]
Now consider the reduction of $J$ modulo $J^{m+1}$. Applying Lemma \ref{l:auxiliary}, we see that
\[
\be(1+\widetilde{x}_1+\widetilde{x}_2,1+\widetilde{y})=1+[\widetilde{x}_1+\widetilde{x}_2,\widetilde{y}] \mod J^{m+1},
\]
\[
\be(1+\widetilde{x}_1,1+\widetilde{y})=1+[\widetilde{x}_1,\widetilde{y}] \mod J^{m+1} \quad\text{and}\quad \be(1+\widetilde{x}_2,1+\widetilde{y})=1+[\widetilde{x}_2,\widetilde{y}] \mod J^{m+1}.
\]

\mbr

Since $[\widetilde{x}_1+\widetilde{x}_2,\widetilde{y}],[\widetilde{x}_1,\widetilde{y}],[\widetilde{x}_2,\widetilde{y}]\in J^m$, it follows that
\[
\be(1+\widetilde{x}_1+\widetilde{x}_2,1+\widetilde{y}) = \be(1+\widetilde{x}_1,1+\widetilde{y}) \cdot \be(1+\widetilde{x}_2,1+\widetilde{y}) \mod J^{m+1},
\]
or equivalently
\[
\be(1+\widetilde{x}_1+\widetilde{x}_2,1+\widetilde{y}) \cdot \be(1+\widetilde{x}_1,1+\widetilde{y})^{-1} \cdot \be(1+\widetilde{x}_2,1+\widetilde{y})^{-1} \in 1+J^{m+1}.
\]
Since the last expression belongs to $(1+J,1+J)$ by definition, we see that it in fact belongs to $(1+J,1+J)\cap(1+J^{m+1})$. By virtue of Proposition \ref{p:halasi}, the last intersection is equal to $(1+J,1+J^m)$; thus
\[
\be(1+\widetilde{x}_1+\widetilde{x}_2,1+\widetilde{y}) \cdot \be(1+\widetilde{x}_1,1+\widetilde{y})^{-1} \cdot \be(1+\widetilde{x}_2,1+\widetilde{y})^{-1} \in (1+J,1+J^m).
\]
Applying the homomorphism $\vp:J\rar{}A$ to the last containment yields \eqref{e:1}.

\mbr

The proof of the second assertion of Proposition \ref{p:commutator-pairing}(b) is almost identical to the proof we just presented, so we will skip it.

\mbr

Finally, the proof of the third assertion is similar, but contains an additional idea. Again, the assertion amounts to the following statement: if $x\in A$, $y\in A^{m-1}$ and $\la\in\fk$, then
\begin{equation}\label{e:2}
\be(1+\la x,1+y)\cdot\be(1+x,1+\la y)^{-1} \in (1+A,1+A^m).
\end{equation}

\mbr

To prove this containment, we begin, as before, by writing $y=\sum_{j=1}^N a_{1,j} a_{2,j} \dotsm a_{(m-1),j}$ for suitable $a_{i,j}\in A$. Consider the  set of formal symbols
\[
X = \bigl\{ \widetilde{x} \bigr\} \cup \bigl\{ \widetilde{a}_{i,j} \bigr\}_{1\leq i\leq m-1,\ 1\leq j\leq N},
\]
and let $R=\bZ[\widetilde{\la}]$ be the polynomial ring over $\bZ$ in one variable $\widetilde{\la}$. There is a unique ring homomorphism $R\rar{}\fk$ such that $\widetilde{\la}\mapsto\la$; by means of this homomorphism we can view $\fk$, and hence also $A$, as $R$-algebras.

\mbr

Next let $J=Fr_R(n,X)$ be the free associative nilpotent algebra of nilpotence class $n$ generated by the set $X$ over $R$ (where $n$ is the nilpotence class of $A$) and let $\vp:J\rar{}A$ be the homomorphism of $R$-algebras determined by $\vp(\widetilde{x})=x$ and $\vp(\widetilde{a}_{i,j})=a_{i,j}$ for all $i,j$.

\mbr

From this point on the proof proceeds as before. Using Lemma \ref{l:auxiliary} we deduce that
\[
\be(1+\widetilde{\la}\widetilde{x},1+\widetilde{y}) = \be(1+\widetilde{x},1+\widetilde{\la}\widetilde{y}) \mod J^{m+1},
\]
which implies that
\[
\be(1+\widetilde{\la}\widetilde{x},1+\widetilde{y}) \cdot \be(1+\widetilde{x},1+\widetilde{\la}\widetilde{y})^{-1} \in (1+J,1+J)\cap(1+J^{m+1}) = (1+J,1+J^m),
\]
where the last equality follows from Proposition \ref{p:halasi}. Applying $\vp$ to the last containment yields \eqref{e:2}, which completes the proof of Proposition \ref{p:commutator-pairing}.

\section{Representations of totally disconnected groups}\label{s:smooth-induction}

This section is mostly a review of standard facts and definitions from the theory of smooth representations of locally compact totally disconnected topological groups. In \S\ref{ss:rodier} we reproduce some important results of F.~Rodier \cite{rodier}, which play a crucial role in Section \ref{s:proof} below.

\subsection{Representations of $\ell$-groups}
In this article we will follow the terminology of \cite{bz}. In particular, we make the following

\begin{defin}
An \emph{$\ell$-group}\footnote{Another common term is ``l.c.t.d. group,'' which stands for ``locally compact totally disconnected group.''} is a Hausdorff topological group $G$ such that the unit element $1\in G$ has a neighborhood basis consisting of compact open subgroups of $G$.
\end{defin}

A primary example of an $\ell$-group for us is the group $\bG(F)$ of $F$-points of an algebraic group $\bG$ over a local nonarchimedean field $F$. The topology on $\bG(F)$ is induced by the usual topology on $F$. An $\ell$-group of the form $\bG(F)$ is also second countable, which will be important for us.

\mbr

It is convenient to isolate a special class of $\ell$-groups that plays an important role in \cite{rodier} (on which our work relies heavily), but for which no term was introduced in \emph{op.~cit.}

\begin{defin}
An \emph{$\ell_c$-group} is an $\ell$-group $G$ that is a \emph{filtered} union of its compact open subgroups; in other words, every $g\in G$ is contained in a compact open subgroup of $G$, and any two such subgroups are together contained in a third such subgroup.
\end{defin}

For instance, if $\bG$ is a unipotent algebraic group over a local nonarchimedean field $F$, then $\bG(F)$ is an $\ell_c$-group. In particular, groups of the form $1+A$, where $A$ is a finite dimensional associative nilpotent algebra over $F$, are $\ell_c$-groups.

\mbr

For convenience, we recall that a \emph{smooth\footnote{The adjective ``algebraic'' is used in \cite{bz} in place of ``smooth.''} representation} of an $\ell$-group $G$ is a pair $(\pi,V)$ consisting of a complex vector space $V$ and a homomorphism $\pi:G\rar{}GL(V)$ with the property that for each $v\in V$, the stabilizer $G^v=\bigl\{g\in G\st \pi(g)v=v\bigr\}$ is open in $G$. By a standard abuse of notation we will sometimes denote a smooth representation $(\pi,V)$ by the single letter $\pi$ or $V$.

\mbr

A \emph{morphism} between smooth representations of $G$ is defined as a homomorphism of abstract representations, and the category of smooth representations of $G$ will be denoted by $\cR(G)$.

\mbr

A smooth representation $(\pi,V)$ of $G$ is said to be:

\begin{itemize}
\item \emph{irreducible} if $V\neq 0$ and the only $\pi(G)$-invariant subspaces of $V$ are $0$ and $V$;
\item \emph{admissible} if for every compact open subgroup $K\subset G$ the subspace \[V^K=\bigl\{v\in V \st \pi(g)v=v\ \forall\,g\in K\bigr\}\] is finite dimensional;
\item \emph{unitarizable} if $V$ has a positive definite Hermitian inner product invariant under $\pi(G)$.
\end{itemize}

By an argument of Jacquet (cf.~\cite[2.11]{bz}), Schur's lemma holds for second countable $\ell$-groups. More precisely, if $(\pi,V)$ is an irreducible smooth representation of a second countable $\ell$-group, then every linear operator $V\rar{}V$ that commutes with $\pi(G)$ is a scalar.

\mbr

In Theorem \ref{t:gutkin-unitary} we need to deal with unitary representations (which are typically not smooth) of topological groups. Let us recall that a \emph{unitary representation} of a topological group $G$ is a pair $(\pi,\cH)$ consisting of a Hilbert space $\cH$ over $\bC$ and a continuous homomorphism $\pi:G\rar{}U(\cH)$, where $U(\cH)$ is the group of unitary linear automorphisms of $\cH$ equipped with the strong operator topology. In this case $(\pi,\cH)$ is said to be \emph{irreducible} if $\cH\neq 0$ and the only \emph{closed} subspaces of $\cH$ invariant under $\pi(G)$ are $0$ and $\cH$.

\subsection{The smooth dual} If $(\pi,V)$ is an \emph{abstract} representation\footnote{That is, $\pi:G\rar{}GL(V)$ is a homomorphism of abstract groups.} of an $\ell$-group $G$, the subset $V^{sm}\subset V$ consisting of vectors $v\in V$ for which the stabilizer $G^v$ is open is a $G$-subrepresentation of $V$, which by definition is smooth. In fact, the functor $V\longmapsto V^{sm}$ from the category $\Rep(G)$ of abstract representations of $G$ to the category $\cR(G)$ is right adjoint to the inclusion functor $\cR(G)\into\Rep(G)$.

\mbr

If $V$ is a smooth representation of $G$, form the contragredient representation $\Hom_{\bC}(V,\bC)$ of $G$ as an abstract group. The smooth part $V^\vee:=\Hom_{\bC}(V,\bC)^{sm}$ is called the \emph{smooth dual} of $V$. We also write\footnote{The notation $(\widetilde{\pi},\widetilde{V})$ is used in \cite{bz} in place of $(\pi^\vee,V^\vee)$, but we prefer not to use the tilde symbol to avoid confusion with its other common uses.} $\pi^\vee:G\rar{}GL(V^\vee)$ for the corresponding homomorphism.

\mbr

The relation between the notion of admissibility and smooth duality is explained by the following standard (and easy) statement.

\begin{lem}\label{l:admissible-duality}
If $\pi$ is a smooth representation of an $\ell$-group $G$, then $\pi$ is admissible if and only if the canonical morphism $\pi\rar{}(\pi^\vee)^\vee$ is an isomorphism. In particular, if $\pi$ is a smooth irreducible admissible representation of $G$, then $\pi^\vee$ is also irreducible and admissible.
\end{lem}

\subsection{Induction functors}\label{ss:induction-functors} In this paper we use three induction functors for representations of $\ell$-groups. Let $G$ be an $\ell$-group, and let $H\subset G$ be a closed subgroup (in particular, $H$ is also an $\ell$-group).

\subsubsection{Smooth induction}\label{sss:smooth-induction}
Suppose $(\rho,W)$ is a smooth representation of $H$. We can form the induced representation of $\rho$ from $G$ to $H$ in the sense of abstract groups. By definition, its underlying vector space $V_{big}$ consists of all functions $f:G\rar{}W$ that satisfy
\[
f(gh) = \rho(h)\cdot f(g) \qquad\forall\,g\in G,\ h\in H,
\]
and the $G$ action on $V_{big}$ is given by left translation:
\[
(g\cdot f)(g') = f(g^{-1}\cdot g') \qquad\forall\, f\in V_{big}, \ g,g'\in G.
\]
The representation $V_{big}^{sm}$ of $G$ will be denoted by $\Ind_H^G\rho$ or $\Ind_H^G W$. Thus we obtain a functor
\[
\Ind_H^G : \cR(H) \rar{} \cR(G),
\]
called \emph{smooth induction}. Frobenius reciprocity follows formally: $\Ind_H^G$ is right adjoint to the restriction functor $\cR(G)\rar{}\cR(H)$.

\subsubsection{Compact induction}\label{sss:compact-induction}
In the setting of \S\ref{sss:smooth-induction}, the subspace of $V_{big}^{sm}$ consisting of functions $f\in V_{big}^{sm}$ whose support is compact modulo $H$ affords a smooth $G$-subrepresentation of $\Ind_H^G\rho$, which we denote by $\cInd_H^G\rho$. Thus we obtain a functor $\cInd_H^G:\cR(H)\rar{}\cR(G)$ together with a natural transformation $\cInd_H^G\into\Ind_H^G$. One calls $\cInd_H^G$ the functor of \emph{induction with compact supports} (or simply ``compact induction'').

\mbr

We caution the reader that in \cite{rodier}, the notation $\Ind_H^G$ is used for the functor of induction with compact supports, while no notation for the smooth induction functor is introduced.

\begin{rem}\label{r:induction-duality}
Suppose $G$ and $H$ are unimodular $\ell$-groups, which is the case in which we are mostly interested. $($For example, if $\bG$ and $\bH$ are unipotent algebraic groups over a local field $F$, then $\bG(F)$ and $\bH(F)$ are unimodular $\ell$-groups.$)$ Then there is a simple relationship between the functors $\cInd_H^G$ and $\Ind_H^G$. Namely, if $\rho$ is any smooth representation of $H$, there is a natural isomorphism of $G$-representations between $\bigl(\cInd_H^G\rho\bigr)^\vee$ and $\Ind_H^G(\rho^\vee)$.
\end{rem}

A more general statement, where the unimodularity of $G$ or $H$ is not assumed, is proved in \cite[Proposition 2.25(c)]{bz}.

\subsubsection{Unitary induction}\label{sss:unitary-induction} Here we only consider the unimodular case, since it is simpler and it is the only one relevant for us. Let $G$ be a unimodular $\ell$-group, and let $H\subset G$ be a closed unimodular subgroup. Then the coset space $G/H$ has a $G$-invariant measure, which is unique up to scaling. Fix one such measure $\mu$.

\mbr

Suppose $\rho:H\rar{}U(\cH)$ is a unitary representation of $H$, and let us denote the corresponding Hermitian inner product on $\cH$ by $\langle\cdot,\cdot\rangle$. Given a measurable function $f:G\rar{}\cH$ that satisfies $f(gh)=\rho(h)\cdot f(g)$ for all $g\in G$ and all $h\in H$, the function $\langle f,f\rangle : g\longmapsto\langle f(g),f(g)\rangle$ descends to the coset space $G/H$. One defines $\uInd_H^G\rho$ as the unitary representation of $G$ where:

\begin{itemize}
\item the underlying space consists of measurable functions $f:G\rar{}\cH$ such that $f(gh)=\rho(h)\cdot f(g)$ for all $g\in G$ and all $h\in H$ and such that $\int_{G/H} \langle f,f\rangle \,d\mu<\infty$;
\item the inner product is defined by
\[
(f_1,f_2) \longmapsto \int_{G/H} \langle f_1,f_2\rangle \,d\mu;
\]
\item the $G$-action is given by left translation, as before.
\end{itemize}

\subsection{Jacquet functors}
If $G$ is an $\ell$-group and $(\pi,V)$ is a smooth representation of $G$, we denote by $J_G(\pi)$ the quotient of $V$ by the subspace spanned by all elements of the form $\pi(g)\cdot v-v$, where $g\in G$ and $v\in V$. By ``abstract nonsense,'' $J_G$ is a right exact functor from $\cR(G)$ to the category of $\bC$-vector spaces. The following observation goes back to the works of Jacquet and Langlands.

\begin{lem}[see Prop.~2.35(b) in \cite{bz}]\label{l:jacquet-exact}
If $G$ is an $\ell_c$-group, the functor $J_G$ is exact.
\end{lem}

\subsection{Representations of $\ell_c$-groups}
Lemma \ref{l:jacquet-exact} has important consequences for the representation theory of $\ell_c$-groups. Special cases of the following result have been used in the literature, but we include a proof for convenience.

\begin{prop}\label{p:dual-injective}
Let $G$ be an $\ell_c$-group. For any $\pi\in\cR(G)$, the smooth dual $\pi^\vee$ is an injective object in the category $\cR(G)$.
\end{prop}

\begin{proof}
For any $\rho\in\cR(G)$ we have a chain of natural isomorphisms
\[
\Hom_{\cR(G)}(\rho,\pi^\vee) \cong \Hom_{\cR(G)}(\rho\tens_{\bC}\pi,\bC) \cong \Hom_{\bC}\bigl(J_G(\rho\tens_{\bC}\pi),\bC\bigr),
\]
and by Lemma \ref{l:jacquet-exact}, the right hand side is a composition of three exact functors in $\rho$.
\end{proof}

\begin{cor}\label{c:admissible-injective}
If $G$ is an $\ell_c$-group and $\pi\in\cR(G)$ is admissible, then $\pi$ is injective in $\cR(G)$.
\end{cor}

\begin{proof}
Combine Lemma \ref{l:admissible-duality} with Proposition \ref{p:dual-injective}.
\end{proof}

\begin{cor}\label{c:irreducible-quotient}
If $G$ is an $\ell_c$-group such that every smooth irreducible representation of $G$ is admissible, then every nonzero smooth representation of $G$ has an irreducible quotient.
\end{cor}

\begin{proof}
By Zorn's lemma, if $\rho\in\cR(G)$ is finitely generated and $\rho\neq 0$, then $\rho$ has an irreducible quotient (no assumptions on $G$ are required here). Hence if $\rho\in\cR(G)$ is arbitrary with $\rho\neq 0$, then $\rho$ has an irreducible \emph{sub}quotient. But in our situation, all smooth irreducible representations of $G$ are injective by Corollary \ref{c:admissible-injective}, so we see that $\rho$ has an irreducible quotient.
\end{proof}

\subsection{The Pontryagin dual of an $\ell_c$-group}
Let $U$ be an $\ell_c$-group. If $U$ is merely viewed as a topological group, one knows how to define its Pontryagin dual $U^*$; it is a commutative topological group which is canonically identified with the Pontryagin dual of the abelianization $U^{ab}=U\bigl/\overline{(U,U)}$ (here $\overline{(U,U)}$ is the closure of the group-theoretic commutator subgroup of $U$).

\mbr

However, for $\ell_c$-groups the Pontryagin dual construction has some special properties:

\begin{lem}\label{l:pontryagin}
If $U$ is an $\ell_c$-group, then every continuous homomorphism $\chi:U\rar{}\cst$ takes values in the unit circle $S^1\subset\cst$ and has open kernel. The Pontryagin dual $U^*$, with its standard compact-open topology, is also an $\ell_c$-group.
\end{lem}

Because of this lemma, in what follows, if $U$ is an $\ell_c$-group we will think of elements of the Pontryagin dual $U^*$ as homomorphisms $\chi:U\rar{}\cst$ whose kernel is open in $U$, and we will refer to them as \emph{smooth homomorphisms}.

\begin{proof}
If $\chi:U\rar{}\cst$ is a continuous homomorphism, then $\chi(U)$ is equal to the union of its compact subgroups, whence $\chi(U)\subset S^1$. Moreover, $\chi$ has open kernel since a sufficiently small open neighborhood of $1\in\cst$ contains no subgroups of $\cst$ other than $\{1\}$. For the last statement we may assume, after replacing $U$ with $U^{ab}$, that $U$ is commutative. Given any compact open subgroup $K\subset U$, let $K^\perp$ denote its annihilator in $U^*$. Then $K^\perp$ is a compact open subgroup of $U^*$ (it can be identified with the Pontryagin dual of the discrete quotient $U/K$). Subgroups of the form $K^\perp$ form a basis of neighborhoods of $1\in U^*$ (here the assumption that $U$ is a \emph{filtered} union of its compact open subgroups becomes crucial), and they also exhaust all of $U^*$ as $K$ gets smaller and smaller, by the previous remarks. This shows that $U^*$ is an $\ell_c$-group.
\end{proof}

\subsection{Summary of some results of \cite{rodier}}\label{ss:rodier}
We conclude this section by stating some of the results appearing in \cite{rodier} that are used in \S\ref{s:proof}.
Until the end of the section we fix a second countable $\ell$-group $G$ and a normal closed subgroup $U\subset G$ that is an $\ell_c$-group. The action of $G$ on $U$ by conjugation induces an action of $G$ on $U^*$ by topological group automorphisms.

\mbr

If $(\rho,W)$ is a smooth representation of $U$ and $\chi:U\rar{}\cst$ is a smooth homomorphism, then, following Rodier, we write $\widetilde{W}(\chi)$ or $\widetilde{\rho}(\chi)$ for $J_U(\rho\tens_{\bC}\chi^{-1})$; more concretely, $\widetilde{\rho}(\chi)=\widetilde{W}(\chi)$ is the quotient of $W$ by the subspace spanned by all elements of the form $\rho(u)\cdot w-\chi(u)\cdot w$, where $u\in U$ and $w\in W$. If $S\subset U^*$ is a subset, we say that $(\rho,W)$ has \emph{spectral support contained in $S$} provided the kernel of the natural map
\[
W \rar{} \prod_{\chi\in S} \widetilde{W}(\chi)
\]
is injective. If this holds, then, in particular, $\rho(U)$ is commutative.

\begin{rem}
Rodier assumes in \cite{rodier} that the subgroup $U$ itself is abelian. However, some of his theorems can be extended to the case where $U$ is not necessarily commutative simply by replacing $G$ with $G/\overline{(U,U)}$ and $U$ with $U/\overline{(U,U)}$. As this generalization is important for the applications we have in mind, we will use it to formulate the theorems below.
\end{rem}

Note that if $(\pi,V)$ is a smooth representation of $G$ and $\chi\in U^*$, then $\widetilde{V}(\chi)$ inherits a natural action of the stabilizer $Z_G(\chi)$ of $\chi$ in $G$.

\begin{thm}\label{t:rodier3}
Let $\chi\in U^*$ be such that the orbit $G\cdot\chi$ is locally closed in $U^*$. The functor $\pi\longmapsto\widetilde{\pi}(\chi)$ is an equivalence between the category of smooth representations of $G$ whose restriction to $U$ has spectral support contained in $G\cdot\chi$, and the category of smooth representations of $Z_G(\chi)$ on which $U$ acts via the scalar $\chi$. A quasi-inverse functor is given by $\cInd_{Z_G(\chi)}^G$.
\end{thm}

This reduces at once to the case where $U$ is abelian, which follows from \cite[Thm.~3, p.~186]{rodier} and the remark on pp.~187--188 of \emph{op.~cit.}

\begin{thm}\label{t:rodier4}
Let $(\pi,V)\in\cR(G)$, and let $\chi\in U^*$ be such that its orbit $G\cdot\chi$ is \emph{closed} in $U^*$ and $\pi\bigl\lvert_U$ has spectral support contained in $G\cdot\chi$. If $\widetilde{\pi}(\chi)$ is admissible as a representation of $Z_G(\chi)$, then $\pi$ is admissible as a representation of $G$.
\end{thm}

This reduces at once to the case where $U$ is abelian, which is \cite[Thm.~4, p.~189]{rodier}.

\section{Proof of Gutkin's conjecture}\label{s:proof}

In this section we complete the proofs of Theorems \ref{t:gutkin-unitary} and \ref{t:gutkin-smooth}. The case where $F$ is either $\bR$ or $\bC$ was covered in \S\ref{ss:local-char-zero}, and the case where $F$ is finite was covered in \S\ref{ss:finite-fields-general}. On the other hand, the statement of Theorem \ref{t:gutkin-smooth} formally makes sense in the case where $F$ is finite, provided we equip all the groups that appear in it with the discrete topology\footnote{In this setting the statements of Theorems \ref{t:gutkin-unitary} and \ref{t:gutkin-smooth} become essentially equivalent, in particular because for a finite discrete group every representation is smooth, and every irreducible representation is finite dimensional (hence admissible) and unitarizable by a standard averaging argument.}. Since the only property of $F$ that plays a role in the proof presented below is the fact that $F$ is a locally compact totally disconnected self-dual topological field, we do not exclude the case where $F$ is finite for clarity.

\subsection{Setup}
We fix a (topological) field $F$ that is either finite and discrete, or local and nonarchimedean of arbitrary characteristic (with its standard topology). We will begin by proving Theorem \ref{t:gutkin-smooth}, and then indicate the (minor) changes and additions needed to adapt our argument to Theorem \ref{t:gutkin-unitary}.

\mbr

Let $A$ be a finite dimensional associative nilpotent algebra over $F$, and let $\pi:1+A\rar{}GL(V)$ be a smooth irreducible representation of $1+A$. If $\dim_{\bC} V=1$, the statement of Theorem \ref{t:gutkin-smooth} becomes vacuous\footnote{The unitarizability of $\pi$ is equivalent to the statement that a smooth homomorphism $1+A\rar{}\bC^\times$ takes values inside the unit circle in $\bC^\times$, which follows from the fact that $1+A$ is an $\ell_c$-group (cf.~Lemma \ref{l:pontryagin}).}, so we may assume that $\dim_{\bC} V>1$. Our first goal is to realize $\pi$ via induction with compact supports from a subgroup of $1+A$ of the form $1+B$, where $B\subset A$ is an $F$-subalgebra of codimension $1$.

\subsection{Key construction}\label{ss:key-construction} Choose the smallest integer $m\geq 2$ such that the restriction of $\pi$ to $1+A^m$ acts on $V$ by scalars, and let $\ze:1+A^m\rar{}\bC^\times$ denote the corresponding smooth homomorphism. Note that $\ze$ must be invariant under the conjugation action of $1+A$. In particular, Corollary \ref{c:commutator-pairing} yields a pairing
\[
\overline{C}_\ze: \bigl( 1+(A/A^2) \bigr) \times \bigl(
1+(A^{m-1}/A^m)\bigr) \rar{} \cst
\]
induced by the map
\[
(1+A)\times(1+A^{m-1})\rar{}\bC^\times, \qquad (g,h)\longmapsto\ze(ghg^{-1}h^{-1}).
\]
Note that $\overline{C}_\ze$ is not identically $1$, since otherwise $\pi(1+A^{m-1})$ would consist of scalar operators, which would contradict the minimality requirement in the choice of $m$.

\mbr

Note that the groups $1+(A/A^2)$ and $1+(A^{m-1}/A^m)$ can be canonically identified with the additive groups of the $F$-vector spaces $A/A^2$ and $A^{m-1}/A^m$. Therefore the Pontryagin dual of $1+(A^{m-1}/A^m)\cong A^{m-1}/A^m$ also has a canonical $F$-vector space structure: given $\la\in F$ and a smooth homomorphism $f:A^{m-1}/A^m\rar{}\cst$, we define $\la\cdot f:A^{m-1}/A^m\rar{}\cst$ by $(\la\cdot f)(x)=f(\la\cdot x)$. Corollary \ref{c:commutator-pairing}(b) means that the pairing $\overline{C}_\ze$ induces an \emph{$F$-linear map}
\[
\Phi_\ze : (A/A^2) \rar{} (A^{m-1}/A^m)^*.
\]
(The right hand side can be thought of either as the Pontryagin dual or as the $F$-vector space dual of $A^{m-1}/A^m$; the two can be identified because $F$ is a self-dual field.)

\mbr

We see that the linear map $\Phi_\ze$ is not identically zero, whence there exists a $1$-dimensional subspace $L\subset A^{m-1}/A^m$ such that the composition
\begin{equation}\label{e:key-composition}
A/A^2 \xrar{\ \ \Phi_\ze\ \ } (A^{m-1}/A^m)^* \rar{} L^*
\end{equation}
is surjective, where the second map is given by restriction.

\mbr

We let $A_1\subset A$ be the preimage of the kernel of the composition \eqref{e:key-composition}, and we let $U\subset A^{m-1}$ be the preimage of $L$. Note that $A_1$ and $U$ are two-sided ideals (hence also $F$-subalgebras) of $A$.

\begin{lem}\label{l:annihilate-commutator}
We have $U\subset A_1$; equivalently, $\ze$ annihilates the commutator $(1+U,1+U)$.
\end{lem}

\begin{proof}
The assertion is vacuous if $m\geq 3$, so assume that $m=2$. We need to check that the commutator pairing
$\overline{C}_\ze : \bigl( 1+(A/A^2) \bigr) \times \bigl( 1+(A/A^2) \bigr) \rar{} \cst$
is identically $1$ on $1+L$. Using the fact that $\dim_F L=1$, we see that there exists $a\in U$ such that every element of $L$ is the image of $\la\cdot a$ for some $\la\in F$. But if $\la_1,\la_2\in F$, then $\la_1\cdot a$ and $\la_2\cdot a$ commute in $A$, whence $\overline{C}_\ze(1+\la_1 a,1+\la_2 a)=1$. This yields the lemma.
\end{proof}

\subsection{Key auxiliary result}
The next observation is key for applying the theorems of \cite{rodier} in our situation.

\begin{lem}\label{l:key}
Let us keep all the notation of \S\ref{ss:key-construction}.

\begin{enumerate}[$($a$)$]
\item There exists a smooth homomorphism $\chi:1+U\rar{}\cst$ such that $\chi\bigl\lvert_{1+A^m}=\ze$.
\item Any two such homomorphisms are $(1+A)$-conjugate.
\item The stabilizer of any such $\chi$ in $1+A$ is equal to $1+A_1$.
\end{enumerate}
\end{lem}

\begin{proof}
(a) By Lemma \ref{l:annihilate-commutator}, the restriction of $\pi$ to $1+U$ has a $1$-dimensional \emph{sub}quotient. By Proposition \ref{p:dual-injective}, every smooth $1$-dimensional representation of $1+U$ is injective as an object of $\cR(1+U)$. Hence $\pi\bigl\lvert_{1+U}$ has a $1$-dimensional quotient; say it is given by a smooth homomorphism $\chi:1+U\rar{}\cst$. Since $\pi$ acts on $1+A^m$ via the scalar $\ze$, we must have $\chi\bigl\lvert_{1+A^m}=\ze$.

\mbr

(b) Suppose $\chi_1,\chi_2:1+U\rar{}\cst$ are smooth homomorphisms with $\chi_1\bigl\lvert_{1+A^m}=\ze=\chi_2\bigl\lvert_{1+A^m}$. Then $\chi_1^{-1}\cdot\chi_2$ factors through $(1+U)/(1+A^m)\cong 1+L$. By our choice of $L$, there exists $x\in A$ with $\Phi_\ze(\overline{x})\bigl\lvert_L=\chi_1^{-1}\cdot\chi_2$, where $\overline{x}$ is the image of $x$ in $A/A^2$. This means that
\[
\chi_1(1+y)^{-1}\cdot\chi_2(1+y) = \ze\bigl( (1+x)(1+y)(1+x)^{-1}(1+y)^{-1} \bigr)
\]
for all $y\in U$. Multiplying both sides by $\chi_1(1+y)$ and recalling that $\ze=\chi_1\bigl\lvert_{1+A^m}$, we find that
\[
\chi_2(1+y) = \chi_1 \bigl( (1+x)(1+y)(1+x)^{-1} \bigr)
\]
for all $y\in U$, proving (b).

\mbr

(c) A calculation similar to the one we just used shows that if $\chi:1+U\rar{}\cst$ is a smooth homomorphism satisfying $\chi\bigl\lvert_{1+A^m}=\ze$, then an element $1+x\in 1+A$ stabilizes $\chi$ if and only if $\ze\bigl((1+x)(1+y)(1+x)^{-1}(1+y)^{-1}\bigr)=1$ for all $y\in U$. This is equivalent to $x\in A_1$ by the construction of $A_1$.
\end{proof}

\subsection{Proof of Theorem \ref{t:gutkin-smooth}}\label{ss:proof-smooth}
We remain in the setup of \S\ref{ss:key-construction}. By induction on $\dim_F A$, we may assume that all the assertions of Theorem \ref{t:gutkin-smooth} hold for any smooth irreducible representation of $1+A_1$. In particular, by Corollary \ref{c:irreducible-quotient}, the restriction $\pi\bigl\lvert_{1+A_1}$ has an irreducible quotient, say $\rho$. By construction, $\rho(1+U)$ commutes with $\rho(1+A_1)$, and hence, by Schur's lemma, the restriction of $\rho$ to $1+U$ is scalar, say, given by a smooth homomorphism $\chi:1+U\rar{}\cst$.

\mbr

Write $G=1+A$ and note that by Lemma \ref{l:key}(b), the orbit $G\cdot\chi$ is closed in the Pontryagin dual $(1+U)^*$, while by Lemma \ref{l:key}(c), the stabilizer $Z_G(\chi)$ equals $1+A_1$.

\mbr

Since $\rho$ is irreducible, it is admissible by the induction hypothesis. Hence so is the smooth dual $\rho^\vee$ (cf.~Lemma \ref{l:admissible-duality}). Note that $1+U$ acts on $\rho^\vee$ via the scalar $\chi^{-1}$ and $Z_G(\chi^{-1})=Z_G(\chi)=1+A_1$. Hence by Theorems \ref{t:rodier3} and \ref{t:rodier4}, the representation $\cInd_{1+A_1}^{1+A}(\rho^\vee)$ of $1+A$ is irreducible and admissible. By Lemma \ref{l:admissible-duality} and Remark \ref{r:induction-duality}, we have
\[
\Ind_{1+A_1}^{1+A}\rho \cong \Ind_{1+A_1}^{1+A}\bigl((\rho^\vee)^\vee\bigr) \cong \left( \cInd_{1+A_1}^{1+A} (\rho^\vee) \right)^\vee,
\]
and by Lemma \ref{l:admissible-duality}, the latter representation is admissible and irreducible. Hence $\Ind_{1+A_1}^{1+A}\rho$ is admissible and irreducible; in particular, the natural map $\cInd_{1+A_1}^{1+A}\rho\rar{}\Ind_{1+A_1}^{1+A}\rho$ must be an isomorphism. On the other hand, the natural map $\pi\rar{}\Ind_{1+A_1}^{1+A}\rho$ coming from Frobenius reciprocity must also be an isomorphism; in particular, $\pi$ is admissible.

\mbr

Next, by the induction hypothesis, the representation $\rho$ admits a $(1+A_1)$-invariant positive definite Hermitian inner product. Using a translation-invariant measure on $(1+A)/(1+A_1)$ and imitating the construction recalled in \S\ref{sss:unitary-induction} above, we can equip $\pi\cong\cInd_{1+A_1}^{1+A}\rho$ with a $(1+A)$-invariant positive definite Hermitian inner product.

\mbr

Finally, using an obvious transitivity property of the functors $\Ind$ and $\cInd$ and the induction hypothesis, we conclude that all the assertions of Theorem \ref{t:gutkin-smooth} hold for the representation $\pi$ of $1+A$. This completes the proof of Theorem \ref{t:gutkin-smooth}.

\subsection{Proof of Theorem \ref{t:gutkin-unitary}}\label{ss:proof-unitary}
Let us end by explaining how the argument presented above needs to be modified in order to yield the assertion of Theorem \ref{t:gutkin-unitary}. We begin by following the constructions of \S\ref{ss:key-construction} almost verbatim (replacing $V$ with $\cH$); everything there remains valid in the case where $\pi:1+A\rar{}U(\cH)$ is a unitary irreducible representation. In order to complete the induction step we must show that
there exists a unitary representation $\rho$ of $1+A_1$ such that $\pi\cong\uInd_{1+A_1}^{1+A}\rho$.

\mbr

To this end, we will replace references to results of \cite{rodier} with the more classical ``Mackey machine;'' we refer the reader to \cite[Chapter 6]{folland} for a more modern exposition.

\mbr

The restriction of $\pi$ to $1+U$ is a unitary (not necessarily irreducible) representation of $1+U$, which factors through the abelianization $(1+U)^{ab}$ by Lemma \ref{l:annihilate-commutator}. Therefore $\pi$ can be decomposed as a direct integral of $1$-dimensional unitary representations of $1+U$, and this decomposition is determined by an $\cH$-projection-valued measure $P$ on $(1+U)^*$. It is easy to see that in the language of \cite[\S6.4]{folland}, the triple $(\pi,(1+U)^*,P)$ is a ``system of imprimitivity'' on $1+A$. Moreover, since $\pi$ acts on $1+A^m$ via the character $\ze$, we see that $P$ must be supported on the subset $S\subset (1+U)^*$ consisting of all $\chi$ such that $\chi\bigl\lvert_{1+A^m}=\ze$. By Lemma \ref{l:key}(b), $S$ is a single $(1+A)$-orbit. Thus $(\pi,S,P)$ is a ``transitive system of imprimitivity'' (\emph{op.~cit.}, \S6.5).

\mbr

Finally, Theorem 6.31 in \emph{op.~cit.} implies that if $\chi\in S$, then $\pi$ can be obtained by unitary induction from a unitary representation of $Z_{1+A}(\chi)=1+A_1$ (cf.~Lemma \ref{l:key}(c)).

\begin{rem}\label{r:generality}
The proof we presented here only relies on the fact that $F$ is self-dual and not on the fact that it is totally disconnected. Thus we do obtain a uniform proof of Theorem \ref{t:gutkin-unitary} that is valid for all self-dual locally compact topological fields at once, as promised.
\end{rem}

\subsection{Final remarks}
Under the hypotheses of Theorem \ref{t:rodier4}, Rodier proved in
\cite[Thm.~5, p.~190]{rodier} that if $\rho$ is a \emph{unitary}
irreducible representation of $Z_G(\chi)$ on which $U$ acts via $\chi$, then \[
\left[ \uInd_{Z_G(\chi)}^G \rho \right]^{sm} =
\cInd_{Z_G(\chi)}^G(\rho^{sm}).
\]
(As before, the reduction to the case where $U$ is abelian, which is treated in \emph{loc.~cit.}, is immediate.)

\mbr

Using this result and analyzing the arguments of \S\S\ref{ss:proof-smooth}--\ref{ss:proof-unitary} above, we see that they imply

\begin{prop}
Let $F$ be a local nonarchimedean field, let $A$ be a finite dimensional associative nilpotent algebra over $F$, and let $\pi:1+A\rar{}U(\cH)$ be a unitary irreducible representation. Then the smooth part $\pi^{sm}$ is irreducible as a representation of $1+A$ $($in the algebraic sense$)$, and $\pi$ $($equivalently, $\pi^{sm}${}$)$ is admissible.
\end{prop}

It is not known to us whether the same result holds for more general $\ell$-groups of the form $\bG(F)$, where $F=\bF_q((t))$ and $\bG$ is a unipotent algebraic group over $F$.

\mbr

If $F$ is an infinite countable field (such as $\bQ$), we can equip $F$ with the discrete topology, so that it becomes a locally compact second countable topological field. Even though such an $F$ is not self-dual, the statements of Theorems \ref{t:gutkin-unitary} and \ref{t:gutkin-smooth} make sense in this setting as well. Justin Conrad asked whether their conclusions still hold. The answer is unknown to us. (Note, however, that they do hold when $A$ is commutative, since Schur's lemma can be applied in this case.)


\begin{thebibliography}{Hal04}

\bibitem[AR05]{adler-roche} J.~Adler and A.~Roche, {\em Discrete series
representations of unipotent $p$-adic groups}, J. Lie Theory
\textbf{15} (2005), no.~1, 261--267.

\bibitem[And98]{andre} C.A.M.~Andr\'e, {\em Irreducible characters
of finite algebra groups}, in: ``Matrices and Group Representations,
Coimbra, 1998'', Textos Mat. S\'er B \textbf{19}, Univ. Coimbra,
Coimbra, 1999, pp.~65--80.

\bibitem[BZ76]{bz} J.N.~Bernstein and A.V.~Zelevinsky, {\em Representations of the group $GL(n,F)$,
where $F$ is a local non-Archimedean field}, Uspehi Mat. Nauk
\textbf{31} (1976), no.~3(189), 5--70.

\bibitem[BD06]{intro} M.~Boyarchenko and V.~Drinfeld, {\em A motivated
introduction to character sheaves and the orbit method for unipotent
groups in positive characteristic}, Preprint, September 2006, arXiv:
{\tt math.RT/0609769}, version 1.

\bibitem[vD78]{van-dijk} G.~van Dijk, {\em Smooth and admissible
representations of $p$-adic unipotent groups}, Compositio Math.
\textbf{37} (1978), no.~1, 77--101.

\bibitem[Fo95]{folland} G.B.~Folland, ``A course in abstract
harmonic
analysis'', Studies in Advanced Mathematics, CRC Press, Boca
Raton,
FL, 1995.

\bibitem[Gu73]{gutkin} E.A.~Gutkin, {\em Representations of
algebraic unipotent groups over a self-dual field}, Funkts. Analiz i
Ego Prilozheniya \textbf{7} (1973), 80.

\bibitem[Hal04]{halasi} Z.~Halasi, {\em On the characters and commutators
of finite algebra groups}, J. Algebra \textbf{275} (2004), 481--487.

\bibitem[Isa95]{isaacs} I.M.~Isaacs, {\em Characters of groups associated
with finite algebras}, J. Algebra \textbf{177} (1995), 708--730.

\bibitem[Ki62]{kirillov} A.A.~Kirillov, {\em Unitary representations of
nilpotent Lie groups}, Uspehi Mat. Nauk \textbf{17} (1962), no. 4
(106), 57--110.

\bibitem[Mo65]{moore} C.C.~Moore, {\em Decomposition of unitary
representations defined by discrete subgroups of nilpotent groups},
Ann. of Math.(2) \textbf{82} (1965), 146--182.

\bibitem[Ro76]{rodier} F.~Rodier, {\em D\'ecomposition spectrale des repr\'esentations
lisses}, Non-commutative harmonic analysis (Actes Colloq.,
Marseille-Luminy, 1976), pp. 177--195. Lecture Notes in Math.,
Vol.~587, Springer, Berlin, 1977.

\end{thebibliography}
\end{document}